\documentclass[a4paper, twoside, 11pt]{article}
\usepackage{titlesec, longtable}

\titleclass{\subsubsubsection}{straight}[\subsection]
\newcounter{subsubsubsection}[subsubsection]

\renewcommand\thesubsubsubsection{\thesubsubsection.\arabic{subsubsubsection}}

\titleformat{\subsubsubsection}
 {\normalfont\normalsize\bfseries}{\thesubsubsubsection}{1em}{}
\titlespacing*{\subsubsubsection}
{0pt}{3.25ex plus 1ex minus .2ex}{1.5ex plus .2ex}

\makeatletter
\renewcommand\paragraph{\@startsection{paragraph}{5}{\z@}%
 {3.25ex \@plus1ex \@minus.2ex}%
 {-1em}%
 {\normalfont\normalsize\bfseries}}
\renewcommand\subparagraph{\@startsection{subparagraph}{6}{\z@}
 {3.25ex \@plus1ex \@minus .2ex}%
 {-1em}%
 {\normalfont\normalsize\bfseries}}
\def\toclevel@subsubsubsection{4}
\def\toclevel@paragraph{5}
\def\toclevel@paragraph{6}
\def\l@subsubsubsection{\@dottedtocline{4}{7em}{4em}}
\def\l@paragraph{\@dottedtocline{5}{10em}{5em}}
\def\l@subparagraph{\@dottedtocline{6}{14em}{6em}}
\@addtoreset{subsubsubsection}{section}
\@addtoreset{subsubsubsection}{subsection}
\@addtoreset{paragraph}{subsubsubsection}
\makeatother

\usepackage{a4wide, fancyhdr, amsmath, amssymb, mathtools, yfonts, bbm}
\usepackage{mathrsfs}
\usepackage{graphicx}
\usepackage{tikz}
\usepackage[all]{xy}
\usepackage[utf8]{inputenc}
\usepackage{amsthm}
\usepackage[english]{babel}
\usepackage{chngcntr}
\usepackage{ifthen}
\usepackage{calc}

\usepackage{authblk}

\newcommand{\Z}{\mathbb{Z}}
\newcommand{\Q}{\mathbb{Q}}

\numberwithin{equation}{section}
\newtheorem{lemma}{Lemma}[section]
\newtheorem{theorem}[lemma]{Theorem}
\newtheorem{prop}[lemma]{Proposition}
\newtheorem{corollary}[lemma]{Corollary}

\newtheorem{definition}[lemma]{Definition}

\newtheorem{question}[lemma]{Question}
\newtheorem{comments}[lemma]{Comments}

\title{\vspace{-\baselineskip}\sffamily\bfseries Binary forms with the same value set II. The case of ${\bf D}_4$}
\author[1]{\'Etienne Fouvry\thanks{CNRS, Laboratoire de math\' ematiques d'Orsay, Universit\' e Paris--Saclay, 91405 Orsay, France, etienne.fouvry@universite-paris-saclay.fr}}
\author[2]{Peter Koymans\thanks{Institute for Theoretical Studies, ETH Zurich, 8092 Zurich, Switzerland, peter.koymans@eth-its.ethz.ch}}
\affil[1]{Universit\'e Paris--Saclay}
\affil[2]{ETH Zurich}
\date{\today}

\begin{document}
\maketitle

\begin{abstract} 
Let $F, G \in \Z[X, Y]$ be binary forms of degree $\geq 3$, non-zero discriminant and with automorphism group isomorphic to $D_4$. If $F(\Z^2) = G(\Z^2)$, we show that $F$ and $G$ are ${\rm GL}(2, \Z)$--equivalent.
\end{abstract}

\section{Introduction} 
This paper is the continuation of \cite{FKC3} and deals with the following question: 

\begin{question}
\label{question} 
Let $d\geq 3$ and let $F(X,Y)$ and $G(X,Y)$ two binary forms of ${\rm Bin}(d, \Q)$ (the set of binary forms with degree $d$, with rational coefficients and with discriminant different from zero) such that $F(\Z^2) = G(\Z^2)$. Does there exist $\gamma := \begin{pmatrix} a & b \\ c & d \end{pmatrix} \in {\rm GL}(2,\Z)$ such that 
\begin{equation}
\label{Fcircgamma}
F(aX + bY, cX + dY) = G(X, Y)?
\end{equation}
If there exists no such matrix $\gamma$, the pair $(F, G)$ is called an extraordinary pair and the form $F$ is called extraordinary.
\end{question}

In order to answer this question, the article \cite{FKC3} shows the crucial importance of the group of automorphisms of $F$ defined as
$$
{\rm Aut}(F, \Q) := \left\{\gamma \in {\rm GL}(2,\Q) : F \circ \gamma = F\right\},
$$
where $F \circ \gamma$ denotes the action of $\gamma$ by linear change of variables (see \eqref{Fcircgamma}). It is known that, for any $F\in {\rm Bin}(d, \Q)$, the group ${\rm Aut}(F, \Q)$ is ${\rm GL}(2, \Q)$--conjugate to one element among the set $\mathfrak K$ of ten subgroups of ${\rm GL}(2, \Z)$ defined by 
\begin{equation}
\label{deffrakK}
\mathfrak K = \{{\bf C}_1, {\bf C}_2, {\bf C}_3, {\bf C}_4, {\bf C}_6, {\bf D}_1, {\bf D}_2, {\bf D}_3, {\bf D}_4, {\bf D}_6\},
\end{equation}
where the letters ${\bf C}_k$ and ${\bf D}_\ell$ respectively correspond to cyclic and dihedral subgroups with cardinality $k$ and $2\ell$. For the definition of these subgroups, we refer the reader to \cite[Lemma 5.1]{FKC3}. In \cite{FKC3}, we treated the cases of ${\bf C}_1, {\bf C}_2, {\bf C}_3, {\bf C}_4, {\bf C}_6, {\bf D}_1, {\bf D}_2$, where we prove that an affirmative answer to Question \ref{question} depends on the existence, in the group of automorphisms, of elements of order $3$ with special type. The cases of ${\bf D}_3$ and ${\bf D}_6$ will be treated in \cite{FKD3D6}, bringing to light a similar characterization but in a more intricate context.

In the present paper, we are concerned with the case of ${\bf D}_4$, which is the dihedral subgroup of ${\rm GL}(2, \Z)$ generated by the two matrices
$$
\begin{pmatrix} 
0 & 1\\ 1 & 0
\end{pmatrix}
\text{ and }
\begin{pmatrix} 0 & 1 \\ -1 & 0
\end{pmatrix}.
$$
This subgroup has order $8$ and every element has order $1$, $2$ or $4$. Write 
$$
G(X,Y) = a_d X^d +a_{d-1} X^{d-1} Y + \cdots +a_1 XY^{d-1} +a_0 Y^d.
$$
Let $G \in {\rm Bin}(d,\Q)$. It is not difficult to show that ${\bf D}_4\subseteq {\rm Aut}(G, \Q)$ if and only if one has the equalities 
\begin{equation}
\label{ak<->ad-k}
a_k = 0 \text{ if } k \equiv 1 \bmod 2 \text{ and } a_k = a_{d - k} \text { for } 0\leq k \leq d.
\end{equation}
In particular, $d$ is even so $d \geq 4$. Furthermore, we have ${\rm Aut}(G, \Q) = {\bf D}_4$, since, in the list $\mathfrak K$ (see \eqref{deffrakK}) there is no group with cardinality strictly divisible by $8$. Finally, every form $F$ satisfying ${\rm Aut}(F, \Q) = \lambda^{-1} {\bf D}_4 \lambda$, with $\lambda \in {\rm GL}(2, \Q)$, has the shape
$$
F =G \circ \lambda,
$$ 
where $G$ satisfies \eqref{ak<->ad-k}. This follows from the general conjugation formula ${\rm Aut}(F \circ \lambda, \Q) = \lambda^{-1} {\rm Aut}(F, \Q) \lambda$.

Our central result is 

\begin{theorem}
\label{centralD4} 
Let $d \geq 4$. Then there is no extraordinary form $F$ such that 
$$
{\rm Aut}(F, \Q) \simeq_{{\rm GL}(2, \Q)} {\bf D}_4.
$$
\end{theorem}

\noindent This theorem was already announced in \cite[Theorem A]{FKC3}. One simple consequence is

\begin{corollary}
Let $F\in {\rm Bin}(4, \Q)$ such that
$$
F(\Z^2) = \{m : m = t^4 + u^4 \textup{ for some } (t, u) \in \Z^2\}.
$$
Then there exists $(a, b, c, d) \in \Z^4$, with $ad - bc = \pm 1$, such that 
$$
F(X, Y) = (aX + cY)^4 + (bX + dY)^4.
$$
\end{corollary}

The corollary follows upon combing Theorem \ref{centralD4} with \cite[Lemma 3.3]{SX}, which asserts that ${\rm Aut}(X^4 + Y^4, \Q) = {\bf D}_4$.

\subsection*{Acknowledgements} 
The first author thanks Michel Waldschmidt for inspiring the thema of this paper, for sharing his ideas and for his encouragements. The second author gratefully acknowledges the support of Dr. Max R\"ossler, the Walter Haefner Foundation and the ETH Z\"urich Foundation.

\section{From extraordinary forms to coverings } 
\subsection{Some definitions}
By a {\it lattice}, we mean an additive subgroup of $\Z^2$ with rank $2$. This lattice is {\it proper} when it is different from $\Z^2$. If $\Lambda$ is a lattice generated by $\vec u, \vec v \in \Z^2$, the {\it index} of $\Lambda$ is the positive integer
\begin{equation}
\label{=det}
[\Z^2 : \Lambda] = \vert \Z^2/\Lambda \vert = \vert \det(\vec u, \vec v) \vert.
\end{equation}

\begin{definition}
\label{defL}
Let $\gamma \in {\rm GL}(2, \Q)$. By definition the lattice associated to $ \gamma$ is the subset $L (\gamma)$ of $\Z^2$ defined by
$$
L(\gamma) := \left\{(x, y)\in \Z^2 : \gamma\begin{pmatrix} x \\ y \end{pmatrix} \in \Z^2\right\}.
$$
\end{definition}

We will use the obvious remarks
\begin{equation}
\label{obvious0}
L(\gamma) = L(-\gamma)
\end{equation}
and
\begin{equation}
\label{obvious}
L(\gamma) = \Z^2 \iff \gamma \text{ has integer coefficients}.
\end{equation}
We recall the following property, which is proved by a direct calculation (see \cite [Lemma 6.9]{FKC3}).

\begin{lemma}
\label{multipleofdet-1} 
Let $\gamma \in {\rm GL}(2, \Q)$. Then $[\Z^2 : L(\gamma)]$ is an integer multiple of $\vert \det \gamma \vert^{-1}$.
\end{lemma}

Several times we will use the next easy lemma, where $v_2(n)$ is the $2$--adic valuation of the integer $n$ with the convention $v_2(0) = + \infty$.

\begin{lemma}
\label{compare2adic}
Let $\alpha, \beta, \gamma \in \Z$ with $\gamma \neq 0$. Suppose that $v_2(\alpha) < v_2(\beta)$ and $v_2(\alpha) < v_2(\gamma)$. Then the lattice defined by the equation
$$
\alpha x_1 + \beta x_2 \equiv 0 \bmod \gamma
$$
is included in the lattice
$$
\{(x_1, x_2) : x_1 \equiv 0 \bmod 2\}.
$$
\end{lemma}
 
\begin{definition}
\label{Def2.2} 
Let $F_1$ and $F_2$ be two forms in ${\rm Bin}(d, \Q)$. By definition, an isomorphism from $F_1$ to $F_2$ is an element $\phi \in {\rm GL}(2, \Q)$ such that $F_1 \circ \phi = F_2$. The set of all such isomorphisms is denoted by ${\rm Isom}(F_1 \rightarrow F_2, \Q)$. Suppose ${\rm Isom}(F_1 \rightarrow F_2, \Q)$ is not empty and let $\rho$ be one of its elements. Then we have the equalities
\begin{align*}
{\rm Isom}(F_1 \rightarrow F_2, \Q) &= \rho \cdot {\rm Aut}(F_2, \Q) = {\rm Aut}(F_1, \Q) \cdot \rho \\
{\rm Isom}(F_2 \rightarrow F_1, \Q) &= \rho^{-1} \cdot {\rm Aut}(F_1, \Q) = {\rm Aut}(F_2, \Q) \cdot \rho^{-1}.
\end{align*}
\end{definition}

We extract from \cite{FKC3} the following key proposition 

\begin{prop}
\label{essential} 
Let $d\geq 3$ and let $(F_1, F_2)$ be a pair of extraordinary forms. Then there exists $\rho \in {\rm GL}(2,\Q)$, a pair of extraordinary forms $(G_1, G_2)$ and a pair $(D, \nu)$ of positive integers such that
\begin{enumerate}
\item we have $F_1 = F_2 \circ \rho$,
\item we have
$$
\Bigl(G_1 \sim_{{\rm GL}(2,\Z)} F_1 \textup{ and } G_2 \sim_{{\rm GL}(2,\Z)} F_2\Bigr) \textup{ or } \Bigl(G_1 \sim_{{\rm GL}(2,\Z)} F_2 \textup{ and } G_2\sim_{{\rm GL}(2,\Z)} F_1\Bigr),
$$
\item we have
\begin{equation}
\label{conditionsforDandnu}
D, \nu \geq 1, D \nu > 1, D \mid \nu \textup{ and } 1 \leq \nu \leq D^2.
\end{equation}
The matrix 
\begin{equation}
\label{defgamma}
\gamma := \begin{pmatrix} D & 0\\ 0 & D/\nu
\end{pmatrix}
\end{equation}
satisfies $G_1 = G_2 \circ \gamma$,
\item we have
\begin{equation}
\label{gammaminimal}
[\Z^2 : L (\gamma)] = \min \left\{[\Z^2 : L( \tau)] : \tau \in {\rm Isom}(G_1 \rightarrow G_2, \Q) \cup {\rm Isom}(G_2 \rightarrow G_1, \Q)\right\},
\end{equation}
\item and finally, we have the two coverings
$$
\Z^2 = \bigcup_{\tau \in {\rm Isom}(G_1 \rightarrow G_2, \Q)} L(\tau) = \bigcup_{\tau \in {\rm Isom}(G_2 \rightarrow G_1, \Q)} L(\tau).
$$
\end{enumerate}
\end{prop}

This proposition essentially gathers the following contents of \cite{FKC3}: Lemma 2.4, \S 7.1, \S 7.2 (particularly the relations (7.7), (7.8), (7.9)) and Proposition 9.1.

\begin{comments}
\label{listofcomments} 
\begin{enumerate}
\item This proposition is quite general, since it requires no assumption concerning the automorphism groups of $F_1$ or $F_2$. 

\item Item {\it 1.} will not be used in the sequel of the proof of Theorem \ref{centralD4}. It recalls the starting point of the construction of $\gamma$ and it implies that both ${\rm Aut}(F_1, \Q)$ and ${\rm Aut}(F_2, \Q)$ are ${\rm GL}(2, \Q)$--conjugate by the formula \eqref{conjugationformula}.

\item The pair $(G_1, G_2) $ is extraordinary with automorphism groups ${\rm GL}(2, \Z)$--conjugate with the automorphism groups of $F_1$ and $F_2$.

\item The initial problem is symmetrical in $(F_1, F_2)$. We eventually break this symmetry in item {\it 4.} to ensure the minimality of the index in \eqref{gammaminimal}. Note that $[\Z^2 : L (\gamma)] = \nu D^{-1}$.

\item Since $\gamma$ belongs to ${\rm Isom}(G_2 \rightarrow G_1, \Q)$, we can explicitly write the isomorphisms appearing in item {\it 5.} in terms of $\gamma$ and ${\rm Aut}(G_1, \Q)$ or ${\rm Aut}(G_2, \Q)$ (see Definition \ref{Def2.2}).

\item In fact, we will only use the first equality written in item 5. The second one will be used in \cite{FKD3D6}.
\end{enumerate}
\end{comments}

\section{About coverings }
\subsection{General notions}
Let $a$, $b$, $c$ and $d$ be four integers. To shorten notations, we write
$$ 
\Z \begin{pmatrix} 
a \\ b
\end{pmatrix} 
+ 
\Z \begin{pmatrix} 
c \\ d 
\end{pmatrix}
=:
\begin{bmatrix} 
a & c \\ b& d
\end{bmatrix}
$$
for the subgroup of $\Z^2$ generated by $(a, b)^T$ and $(c, d)^T$. Recall that if $ad - bc \neq 0$, this subgroup is a lattice. Furthermore, if $\vert ad - bc \vert \geq 2$, this is a proper lattice.

\begin{definition} 
Let $k \geq 1$ be an integer and let $(\Lambda_i)_{1 \leq i \leq k}$ be $k$ lattices. We say that 
$$
\mathcal C = \{\Lambda_1, \ldots, \Lambda_k\}
$$ 
is a covering of $\Z^2$ (or a covering) if and only if
$$
\bigcup_{1 \leq i \leq k} \Lambda_i = \Z^2.
$$
\end{definition}

\begin{definition}[Minimal covering]
Let $k \geq 1$, let $\Lambda_i$ be lattices and let $\mathcal C = \{\Lambda_1, \ldots, \Lambda_k\}$ be a covering. We say that $\mathcal C$ is a minimal covering of length $k$ if and only if replacing any $\Lambda_i$ by some proper sublattice $\Lambda'_i \varsubsetneq \Lambda_i$, the set 
$$
\{\Lambda_1, \ldots, \Lambda_{i - 1}, \Lambda'_i, \Lambda_{i + 1}, \ldots, \Lambda_k\}
$$
is not a covering.
\end{definition}

If $\mathcal C$ is a minimal covering and if $1 \leq i \neq j \leq k$, we never have $\Lambda_i \subseteq \Lambda_j$. In particular, for $k \geq 2$, every $\Lambda_i$ is a proper lattice. The following lemma asserts that from any covering one can extract a minimal covering

\begin{lemma}
\label{existenceminimal} 
Let $k \geq 1$ and let 
$$
\mathcal C := \{\Lambda_1, \ldots, \Lambda_k\}
$$
be a covering. Then there exists an integer $1 \leq k' \leq k$, an injection $\phi: \{1, \ldots, k'\} \rightarrow \{1, \ldots, k\}$, and lattices $\Lambda'_j$ ($1 \leq j \leq k'$), such that 
$$
\mathcal C ':= \{\Lambda'_1, \ldots, \Lambda'_{k'}\}
$$
is a minimal covering, and for all $1 \leq j \leq k'$ one has $\Lambda'_j \subseteq \Lambda_{\phi(j)}$.
\end{lemma}

\begin{proof} 
By reordering and suppressing some $\Lambda_i$ if necessary, we suppose that
\begin{equation}
\label{www}
\bigcup_{1 \leq i \leq s} \Lambda_i = \Z^2 \quad \text{ and } \quad \bigcup_{\substack{1 \leq i \leq s \\ i \neq j}} \Lambda_i \neq \Z^2 \quad \text{ for all } 1 \leq j \leq s
\end{equation}
for some $1 \leq s \leq k$. The case where $s = 1$ is trivial, so we suppose $s \geq 2$. Let 
$$
\mathcal L_s := \{ M_s \text{ lattice}: M_s \subseteq \Lambda_s,\ \Lambda_1 \cup \Lambda_2\cup \dots \cup \Lambda_{s -1} \cup M_s = \Z^2\}.
$$ 
We then have the equality
\begin{equation}
\label{4554}
\Z^2 = \Lambda_1 \cup \dots \cup \Lambda_{s -1} \cup \Lambda'_s,
\end{equation}
with $\Lambda'_s := \bigcap_{M_s \in \mathcal L_s} M_s$. The set $\Lambda'_s$ is a subgroup of $\Z^2$. Its rank can not be $0$ or $1$, because we would have the equality $\cup_{1 \leq i \leq s - 1} \Lambda_i = \Z^2$, which contradicts \eqref{www}. So $\Lambda'_s$ is a lattice and, by construction, it is the smallest lattice included in $\Lambda_s$ and satisfying \eqref{4554}. We continue the same process for the covering 
$$
\{\Lambda_1, \Lambda_2, \ldots ,\Lambda_{s - 1}, \Lambda'_s\}
$$
to replace $\Lambda_{s - 1}$ by a smaller lattice $\Lambda'_{s - 1}$. By iterating this process, we prove the existence of lattices $(\Lambda'_1, \dots, \Lambda'_s)$ with the following three properties
\begin{enumerate}
\item (inclusion) $\Lambda'_i \subseteq \Lambda_i$ for $1 \leq i \leq s$,
\item (covering) $\bigcup_{1 \leq i \leq s} \Lambda'_i = \Z^2$,
\item (partial minimality) let $1 \leq k \leq s $ and let $M_k$ be a lattice such that $M_k \subseteq \Lambda'_k$ and such that 
\begin{equation}
\label{minimality}
\Lambda_1 \cup \dots \cup \Lambda_{k - 1} \cup M_k \cup \Lambda'_{k + 1} \cup \dots \cup \Lambda'_s = \Z^2,
\end{equation}
then $M_k= \Lambda'_k$. 
\end{enumerate}

We now prove that $(\Lambda'_1, \ldots, \Lambda'_s)$ is a minimal covering associated with $\{\Lambda_1, \ldots, \Lambda_s\}$. So we consider lattices $M_i$ such that $M_i \subseteq \Lambda'_i$ and such that 
\begin{equation}
\label{4577}
\Z^2 = M_1 \cup \cdots \cup M_s.
\end{equation}
Our aim is to show that $M_i= \Lambda'_i$. The equality \eqref{4577} implies
$$
\Z^2 = \Lambda_1 \cup \cdots \cup \Lambda_{s -1} \cup M_s.
$$
Therefore we deduce from the partial minimality property (see \eqref{minimality}) that $M_s = \Lambda'_s$. We now have the equality
$$
\Z^2 = M_1 \cup M_2 \cup \cdots \cup M_{s -1} \cup \Lambda'_s,
$$
which implies
$$
\Z^2 = \Lambda_1 \cup \Lambda_2 \cup \cdots \cup \Lambda_{s -2}\cup M_{s - 1} \cup \Lambda'_s.
$$
From the partial minimality property (see \eqref{minimality}) we obtain the equality $M_{s - 1} = \Lambda'_{s -1}$. The end of the proof is by induction.
\end{proof}

\subsection{Minimal coverings with length at most $4$} 
We now list all the minimal coverings with length bounded by $4$.

\begin{theorem}
\label{listofcoverings}
The following is a complete list (up to permutation) of the minimal coverings of $\Z^2$ of length at most four
\begin{align}
\label{ec1}
\left\{
\begin{bmatrix}
1 & 0 \\
0 & 1
\end{bmatrix}
\right\}
\end{align}
\begin{align}
\label{ec2}
\left\{
\begin{bmatrix}
1 & 0 \\
0 & 2
\end{bmatrix},
\begin{bmatrix}
2 & 0 \\
0 & 1
\end{bmatrix},
\begin{bmatrix}
1 & 0 \\
1 & 2
\end{bmatrix}
\right\}
\end{align}
\begin{align}
\label{ec3}
\left\{
\begin{bmatrix}
1 & 0 \\ 
0 & 2
\end{bmatrix},
\begin{bmatrix}
4 & 0 \\ 
0 & 1
\end{bmatrix},
\begin{bmatrix}
1 & 0 \\ 
1 & 2
\end{bmatrix},
\begin{bmatrix}
2 & 0 \\ 
1 & 2
\end{bmatrix}
\right\}
\end{align}
\begin{align}
\label{ec4}
\left\{
\begin{bmatrix}
1 & 0 \\ 
0 & 4
\end{bmatrix},
\begin{bmatrix}
2 & 0 \\ 
0 & 1
\end{bmatrix},
\begin{bmatrix}
1 & 0 \\ 
1& 2
\end{bmatrix},
\begin{bmatrix}
1 & 0 \\ 
2 & 4
\end{bmatrix}
\right\}
\end{align}
\begin{align}
\label{ec5}
\left\{
\begin{bmatrix}
1 & 0 \\ 
0 & 2 
\end{bmatrix},
\begin{bmatrix}
2 & 0 \\ 
0 & 1
\end{bmatrix},
\begin{bmatrix}
1 & 0 \\ 
1 & 4
\end{bmatrix},
\begin{bmatrix}
1 & 0 \\ 
3 & 4 
\end{bmatrix}
\right\}
\end{align}
\begin{align}
\label{ec6}
\left\{
\begin{bmatrix}
1 & 0 \\
0 & 3
\end{bmatrix},
\begin{bmatrix}
3 & 0 \\ 
0 & 1
\end{bmatrix},
\begin{bmatrix}
1 & 0 \\ 1 & 3
\end{bmatrix},
\begin{bmatrix}
1 & 0 \\ 
2 & 3
\end{bmatrix}
\right\}.
\end{align}
\end{theorem}

\begin{proof} 
By \cite[Lemma 6.6 \& 6.7]{FKC3}, there is no minimal covering with length two and only one with length three. So we are left with finding all the minimal coverings with length four.

\vskip .3cm \noindent $\bullet$ {\it Checking that the sets of lattices \eqref{ec3}, \eqref{ec4}, \eqref{ec5} and \eqref{ec6} are coverings.} 
One way to prove this is to give explicit equations of these lattices. Let us focus on \eqref{ec5} since the other cases are similar. The equations of the four corresponding lattices are respectively
$$
y \equiv 0 \bmod 2, \, x\equiv 0 \bmod 2, \, 3x+y \equiv 0 \bmod 4 , \, x+y \equiv 0 \bmod 4.
$$
It remains to check that any pair $(a, b)$ of congruence classes modulo $4$ satisfies at least one of the four equations above.

\noindent $\bullet$ {\it Construction of the minimal coverings.}
Let $L_1, \dots, L_4$ be such that
\[
L_1 \cup L_2 \cup L_3 \cup L_4 = \Z^2.
\]
Ruling out the trivial covering (\ref{ec1}), we may assume that $L_1$, $L_2$, $L_3$ and $L_4$ are all proper subgroups of $\Z^2$. The argument will proceed by repeatedly looking at points outside of $L_1 \cup L_2 \cup L_3 \cup L_4$ and then analyzing in which possible $L_i$ such a point can be. For convenience, these points will be chosen with small coordinates.

By permuting the $L_i$ if necessary, we may assume without loss of generality that
\[
\begin{pmatrix}
1 \\
0
\end{pmatrix}
\in L_1.
\]
Since $L_1 \neq \Z^2$, we see that $\begin{pmatrix} 0 \\ 1 \end{pmatrix} \not \in L_1$. By permuting the $L_i$ again if necessary, we may assume without loss of generality that
\[
\begin{pmatrix}
0 \\
1
\end{pmatrix}
\in L_2.
\]
Finally, a similar argument shows that we may assume without loss of generality that
\[
\begin{pmatrix}
1 \\
1
\end{pmatrix}
\in L_3.
\]
Summarizing, we have so far
\[
\begin{pmatrix}
1 \\
0
\end{pmatrix}
\in L_1, \quad
\begin{pmatrix}
0 \\
1
\end{pmatrix}
\in L_2, \quad
\begin{pmatrix}
1 \\
1
\end{pmatrix}
\in L_3.
\]
We will now consider the vector $(1, -1)^T$, and the proof naturally splits into two cases (case 1 and case 2). Note that we must have $(1, -1)^T \in L_3$ (case 1) or $(1, -1)^T \in L_4$ (case 2).

\subsection*{Case 1}
Let us assume that $(1, -1)^T \in L_3$. Then we have
\[
\begin{pmatrix}
1 \\
0
\end{pmatrix}
\in L_1, \quad
\begin{pmatrix}
0 \\
1
\end{pmatrix}
\in L_2, \quad
L_3 = 
\begin{bmatrix}
1 & 1 \\
1 & -1
\end{bmatrix}.
\]
Note that equality must indeed hold for $L_3$, since the subgroup
\[
\begin{bmatrix}
1 & 1 \\
1 & -1
\end{bmatrix}
\]
has prime index in $\Z^2$ and $L_3 \neq \Z^2$ by assumption. For the rest of the proof, we will often use the above observation implicitly. Looking at the vector $(2, 1)^T$, we get two further cases, namely $(2, 1)^T \in L_2$ (case 1.1) or $(2, 1)^T \in L_4$ (case 1.2), since the cases $(2, 1)^T \in L_1$ or $(2, 1)^T \in L_3$ are impossible (because this forces $L_1 = \Z^2$ respectively $L_3 = \Z^2$).

\subsection*{Case 1.1}
In this case we have $(2, 1)^T \in L_2$ and therefore
\[
\begin{pmatrix}
1 \\
0
\end{pmatrix}
\in L_1, \quad
L_2 =
\begin{bmatrix}
0 & 2 \\
1 & 1
\end{bmatrix}, \quad
L_3 = 
\begin{bmatrix}
1 & 1 \\
1 & -1
\end{bmatrix}.
\]
Considering the vector $(1, 2)^T$ yields two subcases: $(1, 2)^T \in L_1$ (case 1.1.1) or $(1, 2)^T \in L_4$ (case 1.1.2).

\subsection*{Case 1.1.1}
We have
\[
L_1 = 
\begin{bmatrix}
1 & 1 \\
0 & 2
\end{bmatrix}, \quad
L_2 =
\begin{bmatrix}
0 & 2 \\
1 & 1
\end{bmatrix}, \quad
L_3 = 
\begin{bmatrix}
1 & 1 \\
1 & -1
\end{bmatrix},
\]
which corresponds to the covering (\ref{ec2}). 

\subsection*{Case 1.1.2}
Currently, we know that
\[
\begin{pmatrix}
1 \\
0
\end{pmatrix}
\in L_1, \quad
L_2 =
\begin{bmatrix}
0 & 2 \\
1 & 1
\end{bmatrix}, \quad
L_3 = 
\begin{bmatrix}
1 & 1 \\
1 & -1
\end{bmatrix}, \quad
\begin{pmatrix}
1 \\
2
\end{pmatrix} 
\in L_4.
\]
We analyze the possibilities for the vector $(1, -2)^T$. If we add this vector to $L_1$, then the resulting covering is not minimal, as $L_4$ may be removed to obtain the covering (\ref{ec2}). Therefore we arrive at the covering
\[
\begin{pmatrix}
1 \\
0
\end{pmatrix}
\in L_1, \quad
L_2 =
\begin{bmatrix}
0 & 2 \\
1 & 1
\end{bmatrix}, \quad
L_3 = 
\begin{bmatrix}
1 & 1 \\
1 & -1
\end{bmatrix}, \quad
\begin{bmatrix}
1 & 1 \\
2 & -2
\end{bmatrix} \subseteq L_4.
\]
To finish the argument for this case, we consider the vector $(1, 4)^T$. If we add it to $L_4$, then the corresponding covering is not minimal, as $L_1$ may be removed to get the covering (\ref{ec2}). Thus we obtain
\[
L_1 = 
\begin{bmatrix}
1 & 1 \\
0 & 4
\end{bmatrix}, \quad
L_2 =
\begin{bmatrix}
0 & 2 \\
1 & 1
\end{bmatrix}, \quad
L_3 = 
\begin{bmatrix}
1 & 1 \\
1 & -1
\end{bmatrix}, \quad
L_4 = 
\begin{bmatrix}
1 & 1 \\
2 & -2
\end{bmatrix},
\]
which is the covering (\ref{ec4}). 

\subsection*{Case 1.2}
At this point we have
\[
\begin{pmatrix}
1 \\
0
\end{pmatrix}
\in L_1, \quad
\begin{pmatrix}
0 \\
1
\end{pmatrix}
\in L_2, \quad
L_3 = 
\begin{bmatrix}
1 & 1 \\
1 & -1
\end{bmatrix}, \quad
\begin{pmatrix}
2 \\
1
\end{pmatrix}
\in L_4.
\]
We have that $(-2, 1)^T \in L_2$ (case 1.2.1) or $(-2, 1)^T \in L_4$ (case 1.2.2).

\subsection*{Case 1.2.1}
The following information is available to us
\[
\begin{pmatrix}
1 \\
0
\end{pmatrix}
\in L_1, \quad
L_2 = 
\begin{bmatrix}
0 & -2 \\
1 & 1
\end{bmatrix} \quad
L_3 = 
\begin{bmatrix}
1 & 1 \\
1 & -1
\end{bmatrix}, \quad
\begin{pmatrix}
2 \\
1
\end{pmatrix}
\in L_4.
\]
Since the vectors $(2, 1)^T, (1, -2)^T, (1, 2)^T$ together generate $\Z^2$, at least one of $(1, -2)^T$ or $(1, 2)^T$ must be in $L_1$. Thus we get the covering (\ref{ec2}) after removing $L_4$.

\subsection*{Case 1.2.2}
We obtain that
\[
\begin{pmatrix}
1 \\
0
\end{pmatrix}
\in L_1, \quad
\begin{pmatrix}
0 \\
1
\end{pmatrix}
\in L_2, \quad
L_3 = 
\begin{bmatrix}
1 & 1 \\
1 & -1
\end{bmatrix}, \quad
\begin{bmatrix}
2 & - 2 \\
1 & 1
\end{bmatrix}
\subseteq L_4.
\]
We must have that $(1, -2)^T \in L_1$. We will now consider the vector $(4, 1)^T$. We either have $(4, 1)^T \in L_2$, in which case we have the covering (\ref{ec3}), or $(4, 1) \in L_4$, in which case we get the covering (\ref{ec2}) after removing $L_2$. 

\subsection*{Case 2}
We will now assume that $(1, -1)^T \in L_4$. Then we are in the situation
\[
\begin{pmatrix}
1 \\
0
\end{pmatrix}
\in L_1, \quad
\begin{pmatrix}
0 \\
1
\end{pmatrix}
\in L_2, \quad
\begin{pmatrix}
1 \\
1
\end{pmatrix}
\in L_3, \quad
\begin{pmatrix}
1 \\
-1
\end{pmatrix}
\in L_4.
\]
Inspecting the possibilities for the point $(2, 1)^T$, we come to the conclusion that $(2, 1)^T \in L_2$ (case 2.1) or $(2, 1)^T \in L_4$ (case 2.2). 

\subsection*{Case 2.1}
We have arrived at the following configuration
\[
\begin{pmatrix}
1 \\
0
\end{pmatrix}
\in L_1, \quad
L_2 = 
\begin{bmatrix}
0 & 2 \\
1 & 1
\end{bmatrix}, \quad
\begin{pmatrix}
1 \\
1
\end{pmatrix}
\in L_3, \quad
\begin{pmatrix}
1 \\
-1
\end{pmatrix}
\in L_4.
\]
Considering the vector $(1, 2)^T$, we will split into the cases $(1, 2)^T \in L_1$ (case 2.1.1) or $(1, 2)^T \in L_4$ (case 2.1.2). 

\subsection*{Case 2.1.1}
Gathering the information so far, we have
\[
L_1 = 
\begin{bmatrix}
1 & 1 \\
0 & 2
\end{bmatrix}, \quad
L_2 = 
\begin{bmatrix}
0 & 2 \\
1 & 1
\end{bmatrix}, \quad
\begin{pmatrix}
1 \\
1
\end{pmatrix}
\in L_3, \quad
\begin{pmatrix}
1 \\
-1
\end{pmatrix}
\in L_4.
\]
We inspect the different locations for the vector $(3, 1)^T$. If $(3, 1)^T \in L_3$, we obtain the cover (\ref{ec2}) after dropping the lattice $L_4$. Suppose instead that $(3, 1)^T \in L_4$. Then we have
\[
L_1 = 
\begin{bmatrix}
1 & 1 \\
0 & 2
\end{bmatrix}, \quad
L_2 = 
\begin{bmatrix}
0 & 2 \\
1 & 1
\end{bmatrix}, \quad
\begin{pmatrix}
1 \\
1
\end{pmatrix}
\in L_3, \quad
\begin{bmatrix}
1 & 3 \\
-1 & 1
\end{bmatrix}
\subseteq L_4.
\]
At last we consider the vector $(3, -1)^T$. If $(3, -1)^T \in L_3$, we get
\[
L_1 = 
\begin{bmatrix}
1 & 1 \\
0 & 2
\end{bmatrix}, \quad
L_2 = 
\begin{bmatrix}
0 & 2 \\
1 & 1
\end{bmatrix}, \quad
L_3 = 
\begin{bmatrix}
1 & 3 \\
1 & -1
\end{bmatrix}, \quad
L_4 = 
\begin{bmatrix}
1 & 3 \\
-1 & 1
\end{bmatrix},
\]
which corresponds to the covering (\ref{ec5}). If instead $(3, -1)^T \in L_4$, we get the covering (\ref{ec2}) upon removing $L_3$. 

\subsection*{Case 2.1.2}
At this stage we may write
\[
\begin{pmatrix}
1 \\
0
\end{pmatrix}
\in L_1, \quad
L_2 = 
\begin{bmatrix}
0 & 2 \\
1 & 1
\end{bmatrix}, \quad
\begin{pmatrix}
1 \\
1
\end{pmatrix}
\in L_3, \quad
\begin{bmatrix}
1 & 1 \\
-1 & 2
\end{bmatrix}
\in L_4.
\]
Looking at $(3, 1)^T$, we obtain $(3, 1)^T \in L_3$. Then looking at $(3, 2)^T$, we conclude that $(3, 2)^T \in L_1$. Once we discard $L_4$, we get the covering (\ref{ec2}).

\subsection*{Case 2.2}
We know that
\[
\begin{pmatrix}
1 \\
0
\end{pmatrix}
\in L_1, \quad
\begin{pmatrix}
0 \\
1
\end{pmatrix}
\in L_2, \quad
\begin{pmatrix}
1 \\
1
\end{pmatrix}
\in L_3, \quad
L_4 = 
\begin{bmatrix}
1 & 2 \\
-1 & 1
\end{bmatrix}.
\]
Considering the vector $(3, 1)^T$, the argument splits in two cases, namely $(3, 1)^T \in L_2$ (case 2.2.1) and $(3, 1)^T \in L_3$ (case 2.2.2). 

\subsection*{Case 2.2.1}
We have arrived at
\[
\begin{pmatrix}
1 \\
0
\end{pmatrix}
\in L_1, \quad
L_2 = 
\begin{bmatrix}
0 & 3 \\
1 & 1
\end{bmatrix}, \quad
\begin{pmatrix}
1 \\
1
\end{pmatrix}
\in L_3, \quad
L_4 = 
\begin{bmatrix}
1 & 2 \\
-1 & 1
\end{bmatrix}.
\]
This forces $(4, 1)^T \in L_3$ and then $(1, 3)^T \in L_1$, thus giving
\[
L_1 = 
\begin{bmatrix}
1 & 1 \\
0 & 3
\end{bmatrix}, \quad
L_2 = 
\begin{bmatrix}
0 & 3 \\
1 & 1
\end{bmatrix}, \quad
L_3 = 
\begin{bmatrix}
1 & 4 \\
1 & 1
\end{bmatrix}, \quad
L_4 = 
\begin{bmatrix}
1 & 2 \\
-1 & 1
\end{bmatrix}.
\]
This is precisely the covering (\ref{ec6}). 

\subsection*{Case 2.2.2}
Finally, we have to consider the configuration
\[
\begin{pmatrix}
1 \\
0
\end{pmatrix}
\in L_1, \quad
\begin{pmatrix}
0 \\
1
\end{pmatrix}
\in L_2, \quad
L_3 = 
\begin{bmatrix}
1 & 3 \\
1 & 1
\end{bmatrix}, \quad
L_4 = 
\begin{bmatrix}
1 & 2 \\
-1 & 1
\end{bmatrix}.
\]
Considering the vectors $(2, 3)^T$ and $(3, 2)^T$ simultaneously, this gives two cases
\[
L_1 =
\begin{bmatrix}
1 & 3 \\
0 & 2
\end{bmatrix}, \quad
L_2 = 
\begin{bmatrix}
0 & 2 \\
1 & 3
\end{bmatrix}, \quad
L_3 = 
\begin{bmatrix}
1 & 3 \\
1 & 1
\end{bmatrix}, \quad
L_4 = 
\begin{bmatrix}
1 & 2 \\
-1 & 1
\end{bmatrix}
\]
and
\[
L_1 = 
\begin{bmatrix}
1 & 2 \\
0 & 3
\end{bmatrix}, \quad
L_2 = 
\begin{bmatrix}
0 & 3 \\
1 & 2
\end{bmatrix}, \quad
L_3 = 
\begin{bmatrix}
1 & 3 \\
1 & 1
\end{bmatrix}, \quad
L_4 = 
\begin{bmatrix}
1 & 2 \\
-1 & 1
\end{bmatrix}.
\]
In the first case, the lattice $L_4$ is redundant, and we get the covering (\ref{ec2}) upon discarding $L_4$. The last case does not correspond to any non-trivial minimal covering. Indeed, all the $L_i$ have prime index, but their union does not contain the point $(1, 4)^T$. 
\end{proof}

Theoretically speaking, the method employed above leads to the complete list of minimal coverings with length $\leq k$, where $k$ is a given integer. The case $k \leq 6$ is vital for our work \cite{FKD3D6}, which handles the case where the automorphism group is conjugate to ${\bf D}_3$ or ${\bf D}_6$. However, when $k = 5$ and $k = 6$, the number of cases is huge and the method is no longer feasible to execute by hand. This is the reason why we have written algorithms to produce the list of 9 minimal coverings with length equal to 5, and the list of 49 minimal coverings with length equal to 6. To make this paper independent of computer calculations, we have decided to prove Theorem \ref{listofcoverings} by hand, but the results of Theorem \ref{listofcoverings} are rapidly reproduced by our algorithms that will be published in \cite{FKD3D6}.

\section{Proof of Theorem \ref{centralD4}}
\subsection{Preparation of the covering}
The proof of this theorem is accomplished by contradiction by exploiting Proposition \ref{essential}. We start from a pair of extraordinary forms $(F_1, F_2) $ such that ${\rm Aut}(F_1, \Q) \simeq_{{\rm GL}(2, \Q)} {\bf D}_4$. By Comment 2 of \ref{listofcomments} we also have ${\rm Aut}(F_2, \Q) \simeq_{{\rm GL}(2, \Q)} {\bf D}_4$. By item {\it 2.} of Proposition \ref{essential}, we also have
$$
{\rm Aut}(G_1, \Q), \, {\rm Aut}( G_2, \Q) \simeq_{{\rm GL}(2, \Q)} {\bf D}_4
$$ 
thanks to the conjugation formula
\begin{equation}
\label{conjugationformula}
{\rm Aut}(F \circ \lambda, \Q) = \lambda^{-1} {\rm Aut}(F, \Q) \, \lambda, \text{ for all } \lambda \in {\rm GL}(2, \Q) \text{
and for all } F \in {\rm Bin}(d, \Q).
\end{equation}
Recall that 
$$
G_1 \circ \gamma^{-1} = G_2, 
$$
where $\gamma$ is defined in \eqref{defgamma}. We write ${\bf D}_4$ explicitly as 
$$
{\bf D} _4 = \{{\rm id},  A_1, A_2, A_3, -{\rm id}, -A_1, -A_2, -A_3\},
$$
with 
$$
A_1 = \begin{pmatrix} 0 & 1 \\ -1 & 0\end{pmatrix}, A_2 = \begin{pmatrix} 0 & 1 \\ 1 & 0 \end{pmatrix}, \text{ and } A_3 = \begin{pmatrix} 1 & 0 \\ 0 & -1 \end{pmatrix}. 
$$
So we have the equality 
$$
{\rm Aut}(G_2, \Q) = T_2^{-1} {\bf D}_4 T_2,
$$
where $T_2$ is some matrix of ${\rm GL}(2, \Q)$ that we write as
$$
T_2 = \begin{pmatrix} t_1& t_2 \\ t_3 & t_4 \end{pmatrix},
$$
where the $t_i$ are coprime integers. An important integer is
\begin{equation}
\label{defd2}
d_2 := \vert \det T_2 \vert = \vert t_1t_4  - t_2t_3 \vert.
\end{equation}
Thus we split ${\rm Aut}(G_2, \Q)$ into two disjoint sets
$$
{\rm Aut}(G_2, \Q) = \mathfrak S \cup (-\mathfrak S),
$$
with
$$
\mathfrak S := \{{\rm id}, T_2^{-1} A_1T_2, T_2^{-1} A_2T_2, T_2^{-1} A_3T_2\}.
$$
Furthermore, we have ${\rm Isom}(G_1 \rightarrow G_2, \Q) = \gamma^{-1} {\rm Aut}(G_2, \Q)$, by Definition \ref{Def2.2} and we put
\begin{equation}
\label{Lambda(sigma)=}
\Lambda(\sigma) := \{{\bf x} \in \Z^2 : \gamma^{-1} \sigma ({\bf x}) \in \Z^2\} = L(\gamma^{-1} \sigma),
\end{equation}
see Definition \ref{defL}. The equality $\Lambda(-\sigma) = \Lambda(\sigma)$ follows from \eqref{obvious0}. Combining these remarks with the first equality of item {\it 5.} of Proposition \ref{essential}, we have the covering of $\Z^2$
\begin{equation}
\label{eD4cover}
\Z^2 = \bigcup_{\sigma \in \mathfrak S} \Lambda(\sigma)
\end{equation}
by at most four lattices. We will require the explicit equations defining the $\Lambda(\sigma)$. By a direct computation we have

\begin{lemma}
\label{lemma4.1} 
For $\sigma \in \mathfrak S$, the lattice $\Lambda(\sigma)$ is the set of $(x_1, x_2) \in \Z^2$ such that
\begin{equation}
\label{IdD4}
\Lambda({\rm id}):
\begin{cases}
x_1 & \equiv 0 \bmod D\\
\nu x_2 & \equiv 0 \bmod D
\end{cases}
\end{equation}

\begin{equation}
\label{A1}
\Lambda(T_2^{-1}A_1T_2):
\begin{cases}
(t_1t_2+t_3t_4)x_1 + (t_2^2+t_4^2) x_2 & \equiv 0 \bmod d_2D\\
\nu ( (t_1^2+t_3^2)x_1 +(t_1t_2 +t_3t_4)x_2) & \equiv 0 \bmod d_2 D
\end{cases}
\end{equation}

\begin{equation}
\label{A2}
\Lambda(T_2^{-1} A_2 T_2):
\begin{cases}
(t_3t_4-t_1t_2)x_1 +(t_4^2-t_2^2) x_2& \equiv 0 \bmod d_2D\\
\nu ( (t_1^2-t_3^2)x_1 + (t_1t_2-t_3t_4) x_2)& \equiv 0 \bmod d_2 D
\end{cases}
\end{equation}

\begin{equation}
\label{A3}
\Lambda(T_2^{-1} A_3 T_2):
\begin{cases}
(t_2t_3+t_1t_4)x_1 + (2t_2t_4) x_2& \equiv 0 \bmod d_2D\\
\nu ( (2t_1t_3) x_1 + (t_2t_3 +t_1t_4)x_2) & \equiv 0 \bmod d_2D.
\end{cases}
\end{equation}
\end{lemma}

\noindent By \eqref{conditionsforDandnu}, the second equation of \eqref{IdD4} is redundant and thus we have the equality
\begin{equation}
\label{valueofindex}
[\Z^2 : \Lambda({\rm id})] = D.
\end{equation}
Several times we will use the following important lemma.

\begin{lemma}
\label{nondegenerate} 
We adopt the hypotheses of Proposition \ref{essential} and the notations above. Then the covering \eqref{eD4cover} satisfies
\begin{equation}
\label{1479}
\Lambda(\sigma) \neq \Z^2 \textup{ for all } \sigma \in \mathfrak S.
\end{equation}
\end{lemma}

\begin{proof}
For the sake of contradiction, suppose that $\Lambda(\sigma) = \Z^2$ for some $\sigma \in \mathfrak S$. By the definition \eqref{Lambda(sigma)=} and the remark \eqref{obvious}, the matrix associated with $\gamma^{-1} \sigma$ (with $\sigma \in {\rm Aut}(G_2, \Q)$) has integer coefficients. This means that there exists a matrix $M_1$ with integer coefficients such that
\begin{equation}
\label{1564}
G_1 \circ M_1 = G_2.
\end{equation}
We use the minimality of the index of $L(\gamma)$ (see item {\it 4.} of Proposition \ref{essential}) to write 
$$
1 = [\Z^2 : \Lambda(\sigma)] = [\Z^2 : L(\gamma^{-1} \sigma)] \geq [\Z^2 : L( \gamma)] \geq 1.
$$
Again by the remark \eqref{obvious} we deduce that the matrix $M_2$ associated with $\gamma$ (see \eqref{defgamma}), has integer coefficients. So we have the equality
$$
G_2 \circ M_2 = G_1.
$$
Combining with \eqref{1564}, we obtain 
$$
G_1 \circ (M_1 \cdot M_2) = G_2 \circ M_2 = G_1,
$$
which implies that $\vert \det (M_1\cdot M_2)\vert =1$, hence $\vert \det M_1\vert = \vert \det M_2 \vert =1$. So the matrices $M_1$ and $M_2$ belong to ${\rm GL}(2, \Z)$. Therefore the equality \eqref{1564} shows that $G_1$ and $G_2$ are ${\rm GL}(2, \Z)$--equivalent, which is contrary to the hypothesis that $(G_1, G_2)$ is an extraordinary pair (see Comment 3 of \ref{listofcomments}). 
\end{proof}

We now initiate the proof of Theorem \ref{centralD4} by successively restricting the values of $D$. 

\subsection{ {The integer $D$ satisfies $2\leq D \leq 4$} } 
We want to prove the inequality
\begin{equation}
\label{1502}
2\leq D \leq 4.
\end{equation}
We separate our discussion into two subcases based on the value of the union of the lattices $\Lambda(\sigma)$ with $\sigma \neq {\rm id}$.

\subsubsection{ {If the last three lattices do not cover $\Z^2$} } 
We suppose that
\begin{equation}
\label{badunion}
\Z^2 \neq \bigcup_{\sigma \in \mathfrak S - \{{\rm id}\}} \Lambda(\sigma).
\end{equation}
This condition means that $\Lambda({\rm id})$ is essential to obtain the covering \eqref{eD4cover}. By Lemma \ref{nondegenerate} we know that none of the lattices appearing in this covering are trivial. By the existence of a minimal covering (see Lemma \ref{existenceminimal}), there exist four subgroups $\Lambda'_i \ (0 \leq i \leq 3)$ such that 
$$
\begin{cases}
\Lambda'_i \text{ is either } \{0\} \text{ or a proper lattice,} \\ 
\Lambda'_0 \subseteq \Lambda({\rm id}) \text { and } \Lambda'_i \subseteq \Lambda(T_2^{-1} A_i T_2) \ (1 \leq i \leq 3),\\
\cup_{0\leq i \leq 3} \Lambda'_i = \Z^2,\\
\mathcal C := \{\Lambda'_i : \Lambda'_i \neq \{0\}\} \text { is a minimal covering of } \Z^2.
\end{cases}
$$
In particular, by \eqref{eD4cover} and by \eqref{badunion}, we deduce that $\Lambda'_0 \neq \{0\}.$ Thus the covering $\mathcal C$ contains three or four lattices and $\mathcal C$ appears as one of the minimal coverings $\eqref{ec2}, \dots, \eqref{ec6}$ of Theorem \ref{listofcoverings}. Thus the lattice $\Lambda'_0$ is necessarily one of the lattices of these five minimal coverings. By \eqref{=det}, all these lattices have an index equal to $2$, $3$ or $4$. In particular, we have $[\Z^2 : \Lambda'_0] \in \{2, 3, 4\}$. Finally, $[\Z^2 : \Lambda({\rm id})]$ is a divisor (different from $1$) of $[\Z^2 : \Lambda'_0]$. By \eqref{valueofindex}, we deduce the inequality $2 \leq D \leq 4$, and \eqref{1502} is proved. 

\subsubsection{ {If the last three lattices cover $\Z^2$} } 
We now suppose that 
\begin{equation}
\label{eNonCover}
\Z^2 = \bigcup_{\sigma \in \mathfrak{S} - \{{\rm id}\}} \Lambda(\sigma).
\end{equation}
By \eqref{1479}, the equality \eqref{eNonCover} exhibits a non-trivial covering of $\Z^2$ by three lattices. By Theorem \ref{listofcoverings} and by Lemma \ref{existenceminimal} we have the equality 
\begin{equation}
\label{listof3Lambda}
\{\Lambda(T_2^{-1} A_1 T_2), \Lambda(T_2^{-1} A_2 T_2), \Lambda( T_2^{-1} A_3 T_2)\} = \{\Lambda_0, \Lambda_1, \Lambda_2\},
\end{equation}
where the covering \eqref{ec2} is written as $\{\Lambda_0, \Lambda_1, \Lambda_2\}$ respectively. 

\subsubsubsection{ {If \eqref{eNonCover} holds, then the integer $D$ is not divisible by an odd prime $p$} } 
\label{1631}
For the sake of contradiction, suppose that $D$ is divisible by some $p \geq 3$. We will show that there exists at least one $\sigma \in \mathfrak{S} - \{{\rm id}\}$ such that
\begin{equation}
\label{pdividesindex}
p \mid [\Z^2 : \Lambda(\sigma)].
\end{equation}
Such a divisibility is impossible since, on the right--hand side of \eqref{listof3Lambda}, all the lattices have an index equal to $2$. To prove \eqref{pdividesindex}, we will argue by contradiction. So we suppose that 
\begin{equation}
\label{4000} 
p \nmid [\Z^2 : \Lambda(\sigma)] \text{ for all } \sigma \neq {\rm id}. 
\end{equation}
By keeping only the first equation of the systems defining the lattices modulo $p$, we see that the lattices $\Lambda(T_2^{-1} A_1 T_2)$, $\Lambda(T_2^{-1} A_2 T_2)$ and $\Lambda(T_2^{-1} A_3 T_2)$ (see \eqref{A1}, \eqref{A2} and \eqref{A3}) are respectively included in the following lattices $\mathcal L_p$ 
\begin{align}
\mathcal L_p (T_2^{-1} A_1 T_2) &: (t_1t_2+t_3t_4) x_1 + (t_2^2 +t_4^2) x_2 \equiv 0 \bmod p, \label{DEF1}\\
\mathcal L_p (T_2^{-1} A_2 T_2) &: (t_3t_4-t_1t_2) x_1 + (t_4^2 -t_2^2) x_2 \equiv 0 \bmod p, \label{DEF2}\\
\mathcal L_p (T_2^{-1} A_3 T_2) &: (t_2t_3+t_1t_4) x_1 + (2t_2t_4) x_2 \equiv 0 \bmod p. \label{DEF3}
\end{align}
The index of these three lattices is $1$ or $p$. It can not be equal to $p$, otherwise we would have \eqref{pdividesindex} for some $\sigma$. So the index is always equal to $1$, which means that $p$ divides all the coefficients 
\begin{equation}
\label{listofcoef}
(t_1t_2+t_3t_4), \ (t_2^2+t_4^2), \ (t_3t_4-t_1t_2), \ (t_4^2-t_2^2), \ (t_2t_3+t_1t_4) \text{ and } (2t_2t_4).
\end{equation}
This implies that $p$ divides $t_2$ and $t_4$. Let $g := {\rm gcd}(t_2, t_4, p^\infty) \geq p, \tilde{t}_2 = t_2/g, \tilde{t}_4 = t_4/g$. We observe that $g$ divides $d_2$, by its definition \eqref{defd2}. We return to the original system of equations \eqref{A1}, \eqref{A2} and \eqref{A3}, where we keep the first equation of each system. After division by $g$, we observe that the lattices $\Lambda(T_2^{-1} A_1 T_2)$, $\Lambda(T_2^{-1} A_2 T_2)$ and $\Lambda(T_2^{-1} A_3 T_2)$ are respectively included in the following lattices
\begin{align}
\tilde {\mathcal L_p} (T_2^{-1} A_1 T_2) &: (t_1\tilde t_2+t_3\tilde t_4) x_1 \equiv 0 \bmod p, \label{Def1}\\
\tilde{\mathcal L_p }(T_2^{-1} A_2 T_2) &: (t_3\tilde t_4-t_1\tilde t_2) x_1 \equiv 0 \bmod p, \label{Def2}\\
\tilde {\mathcal L_p} (T_2^{-1} A_3 T_2) &: (\tilde t_2t_3+t_1\tilde t_4) x_1 \equiv 0 \bmod p .\label{Def3}
\end{align}
We observe that ${\rm gcd}(t_1, t_3, p) =1$ (otherwise the integers $t_i$ would not be coprime altogether), and that ${\rm gcd}(\tilde t_2, \tilde t_4, p) = 1$. 

-- Suppose that $p$ does not divide $(t_1\tilde t_2+t_3\tilde t_4)$, then the lattice $\tilde{\mathcal L_p}(T_2^{-1} A_1 T_2)$ has its index equal to $p$, and since this lattice contains $\Lambda(T_2^{-1} A_1 T_2)$, we obtain a contradiction with \eqref{4000}.
 
-- Same type of reasoning when $p$ does not divide $(t_3\tilde t_4-t_1\tilde t_2)$. 

-- Now suppose that $p$ divides $ (t_1\tilde t_2+t_3\tilde t_4)$ and $(t_3\tilde t_4-t_1\tilde t_2)$, then $p$ divides $t_1 \tilde t_2$ and $t_3 \tilde t_4$. The above coprimality conditions imply that $p$ does not divide the coefficient of $x_1$ in \eqref{Def3}, contradicting \eqref{4000}. 

So $D$ has no odd prime divisor when \eqref{eNonCover} holds. 

\subsubsubsection{ {If \eqref{eNonCover} holds, then the integer $D$ is not divisible by $8$} } 
The strategy is the same as in \S \ref{1631}. We suppose that $D$ is divisible by $8$. We will prove that there is a $\sigma \in \mathfrak{S} - \{{\rm id}\}$ such that
\begin{equation}
\label{4dividesindex}
4 \mid [\Z^2 : \Lambda(\sigma)].
\end{equation}
Such a divisibility contradicts the equality \eqref{listof3Lambda}, since all the lattices $\Lambda_i$ ($1 \leq i \leq 3$) have index equal to $2$. To prove \eqref{4dividesindex}, we will argue by contradiction. So we suppose that 
\begin{equation}
\label{4001} 
4 \nmid [\Z^2 : \Lambda(\sigma)] \text{ for all } \sigma \neq {\rm id}. 
\end{equation}
The lattices 
$\Lambda(T_2^{-1} A_1 T_2),$ $\Lambda(T_2^{-1} A_2 T_2)$ and $\Lambda(T_2^{-1} A_3 T_2)$ are respectively included in the following lattices
\begin{align*}
\mathcal L_8 (T_2^{-1} A_1 T_2) &: (t_1t_2+t_3t_4) x_1 + (t_2^2 +t_4^2) x_2 \equiv 0 \bmod 8, \\
\mathcal L_8 (T_2^{-1} A_2 T_2) &: (t_3t_4-t_1t_2) x_1 + (t_4^2 -t_2^2) x_2 \equiv 0 \bmod 8, \\
\mathcal L_8 (T_2^{-1} A_3 T_2) &: (t_2t_3+t_1t_4) x_1 + (2t_2t_4) x_2 \equiv 0 \bmod 8. 
\end{align*}
These lattices have an index equal to $1$, $2$, $4$ or $8$. It can not be divisible by $4$ (otherwise, we would contradict \eqref{4001}). So these indexes are equal to $1$ or $2$, which means that $4$ divides all the coefficients listed in \eqref{listofcoef}. This implies that $2$ divides $t_2$ and $t_4$. Let $g = (t_2, t_4, 2^\infty) \geq 2$, $\tilde t_2 = t_2/g$, $\tilde t_4 = t_4/g$. We return to the initial definitions of the lattices \eqref{A1}, \eqref{A2} and \eqref{A3}, where we only keep the first equation. Since $g$ divides $d_2$, we observe that the lattices $\Lambda(T_2^{-1} A_1 T_2)$, $\Lambda(T_2^{-1} A_2 T_2)$ and $\Lambda(T_2^{-1} A_3 T_2)$ are respectively included in the following lattices
\begin{align}
\tilde {\mathcal L_8} (T_2^{-1} A_1 T_2) &: (t_1\tilde t_2+t_3\tilde t_4) x_1 + g (\tilde t_2^2 +\tilde t_4^2) x_2\equiv 0 \bmod 8, \label{Def1*}\\
\tilde{\mathcal L_8 }(T_2^{-1} A_2 T_2) &: (t_3\tilde t_4-t_1\tilde t_2) x_1 + g (\tilde t_4^2 - \tilde t_2^2) x_2\equiv 0 \bmod 8, \label{Def2*}\\
\tilde {\mathcal L_8} (T_2^{-1} A_3 T_2) &: (\tilde t_2t_3+t_1\tilde t_4) x_1 +g (2\tilde t_2 \tilde t_4) x_2 \equiv 0 \bmod 8 .\label{Def3*}
\end{align}
We exploit the coprimality ${\rm gcd}(t_1, t_3, 2) = 1$ (otherwise, the integers $t_i$ would not be coprime altogether) and the coprimality ${\rm gcd} (\tilde t_2, \tilde t_4, 2) = 1$ to deduce that at least one of the three coefficients attached to $x_1$ in \eqref{Def1*}, \eqref{Def2*} and \eqref{Def3*} is not divisible by $4$. This implies that the corresponding lattice $\tilde{\mathcal L_8}(\sigma)$ has index equal to $4$ or $8$. Hence the associated lattice $\Lambda(\sigma)$ (contained in $\tilde{\mathcal L_8}(\sigma)$) has index divisible by $4$. This contradicts our assumption \eqref{4001}. 

We have considered all the possible cases. The proof of \eqref{1502} is complete. 

\subsection{ {The integer $D$ is different from $3$} } 
Our task is to prove that 
\begin{equation}
\label{Ddifferentfrom3}
D \neq 3.
\end{equation}
We will now assume that $D = 3$ to derive a contradiction. We already know that $[\Z^2 : \Lambda({\rm id})] = D = 3$ by \eqref{valueofindex} and $\Lambda(\sigma) \neq \Z^2$ for all $\sigma \in \mathfrak S$ by Lemma \ref{nondegenerate}. Reasoning as before, when we obtained \eqref{pdividesindex}, there exists $\sigma \in \mathfrak S - \{ {\rm id}\}$ such that
$$
3 \mid [\Z^2 : \Lambda(\sigma)].
$$
Then the minimal covering contained in the covering $\{\Lambda(\sigma) : \sigma \in \mathfrak S\}$ (see \eqref{eD4cover}), which exists by Lemma \ref{existenceminimal}, can only be the covering \eqref{ec6} by Theorem \ref{listofcoverings}. Then the covering $\{\Lambda(\sigma) : \sigma \in \mathfrak S\}$ from \eqref{eD4cover} must coincide with \eqref{ec6}.

Write $g = {\rm gcd}(t_2, t_4, 3^\infty)$, $\tilde t_2 = t_2/g$ and $\tilde t_4 = t_4/g$. We now split our discussion according to the classes modulo $3$ of the numbers $g$, $t_2$ and $t_4$.

\subsubsection{ {If $3 \mid g$} }
\label{3dividesg} 
We follow the arguments that led to equation \eqref{pdividesindex}. In particular, consider the three lattices $\tilde{\mathcal L}_p$ defined by \eqref{Def1}, \eqref{Def2} and \eqref{Def3} with $p=3$. At least, one of the coefficients of $x_1$ is non-zero modulo $3$. Thus we get the existence of some $\sigma \in \mathfrak S - \{{\rm id}\}$, such that
$$
\Lambda(\sigma) \subseteq \{(x_1, x_2) : x_1 \equiv 0 \bmod 3\}.
$$
But $\Lambda({\rm id}) = \{(x_1, x_2) : x_1 \equiv 0 \bmod 3\}$ (see \eqref{IdD4} and \eqref{conditionsforDandnu}). So the covering \eqref{eD4cover} can never be equal to \eqref{ec6}, which is the desired contradiction.

\subsubsection{ If $3 \nmid g$, $3 \nmid t_2$, $3 \nmid t_4$ and $t_2 \equiv t_4 \bmod 3$  } 
Under these assumptions, we have the congruences
$$
t_2^2 +t_4^2 \equiv 2 \bmod 3, \ 2t_2t_4 \equiv 2 \bmod 3
$$
and 
$$ 
t_1t_2+t_3t_4 \equiv t_2t_3 +t_1t_4 \bmod 3.
$$
By \eqref{DEF1} and \eqref{DEF3}, we see that the lattices $\mathcal L_3 (T_2^{-1} A_1 T_2)$ and $\mathcal L_3 (T_2^{-1} A_3 T_2)$ coincide. Both have index 3. This implies that the two lattices $\Lambda(T_2^{-1} A_1 T_2)$ and $\Lambda(T_2^{-1} A_3 T_2)$ sit inside the same lattice with index $3$. We obtain the same contradiction as in \S \ref{3dividesg}.

\subsubsection{\ {If $3 \nmid g$, $3 \nmid t_2$, $3 \nmid t_4$ and $t_2 \not\equiv t_4 \bmod 3$} } 
In this situation, we obtain 
$$
t_2^2 +t_4^2 \equiv 2 \bmod 3, 2t_2t_4 \equiv 1 \bmod 3
$$
and 
$$ 
t_1t_2+t_3t_4 \equiv (2t_2t_4) ( t_1t_2 +t_3t_4) \equiv 2 (t_1t_4 +t_2t_3) \bmod 3.
$$
Up to a factor $2$ the equations \eqref{DEF1} and \eqref{DEF3} coincide. Once again, this implies that two of the lattices appearing in the covering \eqref{eD4cover} are sublattices of the same lattice with index $3$. We obtain the same contradiction as above.

\subsubsection{ {If $3 \nmid g$ and if either $3\mid t_2$ or $3 \mid t_4$} }
Consider the lattices $\mathcal L_3 (T_2^{-1} A_1 T_2)$ and $\mathcal L_3 (T_2^{-1} A_2 T_2) $ defined by \eqref{DEF1} and \eqref{DEF2}. Their equations reduce to $\pm (ax_1 + x_2) \equiv 0 \bmod 3$ (for some integer $a$). So these two lattices coincide, and $\Lambda(T_2^{-1} A_1 T_2)$ and $\Lambda(T_2^{-1} A_2 T_2)$ are both included in the same lattice with index $3$. Therefore we obtain a contradiction with the covering \eqref{eD4cover} in this case as well.

The proof of \eqref{Ddifferentfrom3} is complete. Combining \eqref{1502} and \eqref{Ddifferentfrom3}, we have reduced our theorem to the following situation. 

\subsection{ {Study when $D \in \{2, 4\}$} } 
Our first task is to circumscribe the possible values of $\nu$. Recall the conditions \eqref{conditionsforDandnu}. They lead to the following possibilities for $(D, \nu)$: $(2,2)$, $(2,4)$, $(4, 4)$, $(4, 8)$, $(4, 12)$ and $(4, 16)$. We will restrict this list to
\begin{equation}
\label{setforDnu}
(D, \nu) \in \{(2, 2), (2, 4), (4, 4), (4, 8)\}.
\end{equation}
To prove \eqref{setforDnu}, we start from the covering \eqref{eD4cover} with the condition \eqref{1479}. Two possibilities occur

\begin{itemize} 
\item If for some $\sigma^\dag \in \mathfrak S$, one has $\cup_{\sigma \in \mathfrak S -\{ \sigma^\dag\}} \Lambda(\sigma) = \Z^2$. This is a covering by three lattices. By Theorem \ref{listofcoverings}, these three lattices have index $2$. So we have
\begin{equation}
\label{1776}
[\Z^2 : \Lambda(\sigma_0)] = 2 \text{ for some } \sigma_0 \neq {\rm id}.
\end{equation}
\item If such a $\sigma^\dag$ does not exist. The covering \eqref{eD4cover} corresponds to one of the coverings \eqref{ec3}, \eqref{ec4}, \eqref{ec5} or \eqref{ec6} of Theorem \ref{listofcoverings}. Since $\Lambda({\rm id}) $ has index equal to $2$ or $4$, we can eliminate the covering \eqref{ec6} (because the lattices comprising this covering all have an index equal to $3$). Now, in the coverings \eqref{ec3}, \eqref{ec4} and \eqref{ec5}, there are exactly two lattices with index $2$. So \eqref{1776} is also true in that case. 
\end{itemize}

\noindent We now exploit the relation \eqref{gammaminimal}, with the choice $\tau = \sigma_0$, and the relation $[\Z^2 : L(\gamma)] = \nu/D$ by Comment 4 of \ref{listofcomments} to obtain
$$
\nu D^{-1} \leq 2,
$$
which leads to the divisibility
$$
\nu \mid 2D,
$$
since $\nu D^{-1}$ is an integer. This gives the condition \eqref{setforDnu}.
 
We can further restrict the set of possible values for the pair $(D, \nu)$ as follows. By \eqref{Lambda(sigma)=}, we have $\Lambda(\sigma_0) = L (\gamma^{-1} \sigma_0)$, where $\sigma_0$ satisfies \eqref{1776}. But $\vert \det (\gamma^{-1} \sigma_0) \vert = \frac \nu{D^2}$. By Lemma \ref{multipleofdet-1}, we deduce that $2$ is a multiple of $D^2/\nu$. This gives the inequality $D^2 \leq 2\nu$. So we know that the set defined in \eqref{setforDnu} is restricted to
\begin{equation}
\label{newsetforDnu}
(D, \nu) \in \{(2, 2), (2,4), (4, 8)\}. 
\end{equation}

\subsection{ {Possible values for $d_2$} } 
We now investigate the possible values for $d_2$.

\subsubsection{ {The integer $d_2$ has no odd prime divisor} } 
\label{nooddprime} 
We will prove that 
\begin{equation}
\label{1,2,4,8,16}
d_2 \in \{1, 2, 4, 8, 16, \dots\}.
\end{equation}
Suppose that $d_2$ is divisible by some odd prime $p$. We recall that $\nu \in \{2, 4, 8\}$ (see \eqref{newsetforDnu}) and we return to the definitions \eqref{A1}, \eqref{A2} and \eqref{A3} to deduce that the lattices $\Lambda(T_2^{-1} A_1 T_2)$, $\Lambda(T_2^{-1} A_2 T_2)$ and $\Lambda(T_2^{-1} A_3 T_2)$ are respectively included in the following lattices
\begin{align*}
\mathcal M_{1, p} &:
\begin{cases} (t_1t_2+t_3t_4) x_1 + (t_2^2 +t_4^2) x_2\equiv 0 \bmod p,\\
(t_1^2 +t_3^2) x_1 + (t_1t_2+t_3t_4) x_2 \equiv 0 \bmod p,
\end{cases} \\
\mathcal M_{2, p} &:
\begin{cases} (t_3t_4-t_1t_2) x_1 + (t_4^2 -t_2^2) x_2\equiv 0 \bmod p,\\
(t_1^2 -t_3^2) x_1 + (t_1t_2-t_3t_4) x_2 \equiv 0 \bmod p,
\end{cases} \\
\mathcal M_{3, p} &:
\begin{cases} (t_2t_3+t_1t_4) x_1 + (2t_2t_4) x_2\equiv 0 \bmod p,\\
(2t_1t_3) x_1 + (t_2t_3+t_1t_4) x_2 \equiv 0 \bmod p.
\end{cases}
\end{align*}
We use a combinatorial lemma dealing with the antidiagonal coefficients in the above system of equations.

\begin{lemma}
\label{antidiagonal}
Let $p$ be an odd prime. Let $(t_1, t_2, t_3, t_4)\in \Z^4$ be such that $p \nmid {\rm gcd}(t_1, t_2, t_3, t_4)$. Consider the following three $2$--sets of quadratic forms
$$
S_1:= \{t_2^2+t_4^2, t_1^2+t_3^2\}, \quad S_2 := \{t_4^2-t_2^2, t_1^2-t_3^2\}, \quad S_3 := \{2t_2t_4, 2t_1t_3\}.
$$
Then for all pairs $(i, j)$ with $1 \leq i < j \leq 3$, there exists $P \in S_i \cup S_j$ such that $p \nmid P$.
\end{lemma}

\begin{proof} 
Omitted.
\end{proof}

Lemma \ref{antidiagonal} implies that there exist $i$ and $j$ such that $\mathcal M_{i, p}$ and $\mathcal M_{j, p}$ have index divisible by $p$. Hence the indices of the corresponding lattices $\Lambda(T_2^{-1} A_i T_2) \subseteq \mathcal M_{i, p}$ and $\Lambda(T_2^{-1} A_j T_2) \subseteq \mathcal M_{j, p}$ are divisible by $p$. This is incompatible with the covering \eqref{eD4cover} and Theorem \ref{listofcoverings}, since the index of $\Lambda({\rm id})$ is $2$ or $4$. The proof of \eqref{1,2,4,8,16} is complete. 

\subsubsection{ {The case $d_2 = 1$ is impossible} }
\label{d2differentfrom1} 
Recall that $D \in \{2, 4\}$. By keeping only the first equation in the systems \eqref{IdD4}, \eqref{A1}, \eqref{A2} and \eqref{A3} and by replacing $d_2D$ by $2$, we deduce that the four lattices $\Lambda(\sigma)$ ($\sigma\in \mathfrak S)$ are respectively included in the four lattices defined by the equations
\begin{align*}
\mathcal L_2 ({\rm id}) &: x_1 \equiv 0 \bmod 2,\\
\mathcal L_2 (T_2^{-1} A_1 T_2)&: (t_1t_2+t_3t_4) x_1 + (t_2^2 +t_4^2) x_2 \equiv 0 \bmod 2, \\ 
\mathcal L_2 (T_2^{-1} A_2 T_2) &: (t_3t_4-t_1t_2) x_1 + (t_4^2 -t_2^2) x_2 \equiv 0 \bmod 2, \\
\mathcal L_2 (T_2^{-1} A_3 T_2) &: (t_2t_3+t_1t_4) x_1 \equiv 0 \bmod 2.
\end{align*}
The assumption $d_2 = 1$ implies $t_1t_4 + t_2t_3 \equiv 1 \bmod 2$, from which we deduce the equality 
$$
\mathcal L_2({\rm id}) = \mathcal L_2(T_2^{-1} A_3 T_2).
$$
Considerations of parities also lead to the equality 
$$
\mathcal L_2(T_2^{-1} A_1 T_2) = \mathcal L_2(T_2^{-1} A_2 T_2).
$$ 
Again playing with the parities of the $t_i$ and using $1 = d_2 \equiv t_1t_4 + t_2t_3 \bmod 2$, we see that we never have $t_1t_2+t_3t_4 \equiv t_2^2+t_4^2 \equiv 0 \bmod 2$, which means that $\mathcal L_2 (T_2^{-1} A_1 T_2) \neq \Z^2$. These considerations show that the covering \eqref{eD4cover} leads to the non-trivial covering
$$
\Z^2 = \mathcal L_2 ({\rm id}) \cup \mathcal L_2 (T_2^{-1} A_1 T_2),
$$
which is in contradiction with Theorem \ref{listofcoverings}. 

\subsubsection{ {The case $d_2 = 2$ is impossible} } 
\label{d2differentfrom2}
The strategy is the same as in the section above, but more intricate since we will distinguish cases based on the parity of $t_2$ and $t_4$. We suppose that $d_2 =2$ and we will arrive at a contradiction. Since $4 \mid d_2D$, the four lattices $\Lambda(\sigma)$ ($\sigma\in \mathfrak S)$ are respectively included in the four lattices defined by the equations
\begin{align*}
\mathcal L_2 ({\rm id}) &: x_1 \equiv 0 \bmod 2,\\
\mathcal L_4 (T_2^{-1} A_1 T_2) &: (t_1t_2+t_3t_4) x_1 + (t_2^2 +t_4^2) x_2 \equiv 0 \bmod 4, \\ 
\mathcal L_4 (T_2^{-1} A_2 T_2) &: (t_3t_4-t_1t_2) x_1 + (t_4^2 -t_2^2) x_2 \equiv 0 \bmod 4, \\
\mathcal L_4 (T_2^{-1} A_3 T_2) &: (t_2t_3+t_1t_4) x_1 +(2t_2t_4) x_2 \equiv 0 \bmod 4.
\end{align*}
Recall the equality $\vert t_1t_4-t_2t_3 \vert = d_2 = 2$.

\subsubsubsection{ {If $t_2$ and $t_4$ have different parities} } 
We then have the equalities
$$
{\rm gcd}(t_2^2+t_4^2, 4) = {\rm gcd}(t_4^2-t_2^2, 4) = 1,
$$
which implies
$$
[\Z^2 : \mathcal L_4 (T_2^{-1} A_1 T_2)] = [\Z^2 : \mathcal L_4 (T_2^{-1} A_2 T_2)] = 4
$$
by a direct study of the equations defining these lattices. Furthermore, the coefficient of $x_1$ in the equation defining $\mathcal L_4 (T_2^{-1} A_3 T_2)$ satisfies one of the following two relations
\begin{equation}
\label{equiv2mod4}
t_2t_3+t_1t_4 =
\begin{cases} 
d_2 + 2t_2t_3 \equiv 2 \bmod 4 \\
d_2 + 2t_1t_4 \equiv 2 \bmod 4.
\end{cases}
\end{equation}
This implies the equality between lattices
$$ 
\mathcal L_2({\rm id}) = \mathcal L_4 (T_2^{-1} A_3 T_2).
$$
The covering \eqref{eD4cover} leads to the covering
$$
\Z^2 = \mathcal L_2 ({\rm id}) \cup \mathcal L_4 (T_2^{-1} A_1 T_2) \cup \mathcal L_4 (T_2^{-1} A_2 T_2) ,
$$
which is nonsense, since the index of these three lattices are $2$, $4$ and $4$ respectively. Therefore this covering does not correspond to a minimal covering in Theorem \ref{listofcoverings}. 

\subsubsubsection{ {If $t_2$ and $t_4$ are both even} } 
Under this assumption, we have the following similarities between the coefficients of the lattices $\mathcal L_4 (T_2^{-1} A_1 T_2)$ and $\mathcal L_4 (T_2^{-1} A_2 T_2)$ defined in \S \ref{d2differentfrom2}: 
$$
t_2^2+t_4^2 \equiv t_4^2-t_2^2 \equiv 0 \bmod 4,
$$
and 
$$
t_1t_2+t_3t_4 \equiv t_3t_4-t_1t_2\equiv 0 \bmod 2.
$$
We now discuss on the class $t_1t_2+t_3t_4 \equiv t_3t_4-t_1t_2 \bmod 4.$

\paragraph{If $t_2$ and $t_4$ are both even and if $t_1t_2+t_3t_4 \equiv t_3t_4 -t_1t_2 \equiv 2 \bmod 4$.} 
We return to the definitions of the lattices to deduce the equalities between lattices
$$
\mathcal L_2 ({\rm id}) = \mathcal L_4 (T_2^{-1} A_1 T_2) = \mathcal L_4 (T_2^{-1} A_2 T_2),
$$
which certainly can not lead to a covering.

\paragraph{If $t_2$ and $t_4$ are both even and if $t_1t_2+t_3t_4 \equiv t_3t_4 -t_1t_2 \equiv 0 \bmod 4$.} 
In \eqref{equiv2mod4}, we have already seen that
\begin{equation}
\label{t2t3+t1t4=2}
t_2t_3 + t_1 t_4 \equiv 2 \bmod 4.
\end{equation}
We split our discussion according to the value of $D$ (see \eqref{newsetforDnu}).

\vskip .3cm
 
$\diamond$ {\bf Case 1: $D = 4$.} In that case, we have $\nu = 8$ and $d_2D = 8$, so the second equations of \eqref{A1}, \eqref{A2} and \eqref{A3} are automatically satisfied. We observe that 
$$
\Lambda({\rm id}) = \{(x_1, x_2) : x_1 \equiv 0 \bmod 4\}. 
$$
By \eqref{t2t3+t1t4=2}, we have the inclusion
$$
\Lambda(T_2^{-1} A_3 T_2) \subseteq \Lambda({\rm id}).
$$
The covering \eqref{eD4cover} is simplified to the covering
$$
\Z^2 = \Lambda({\rm id}) \cup \Lambda(T_2^{-1} A_1 T_2) \cup \Lambda(T_2^{-1} A_2 T_2),
$$
where the first lattice has index $4$ and where the last two lattices have an index $\geq 2$. This covering with three lattices does not resonate with Theorem \ref{listofcoverings}.

\vskip .3cm

$\diamond $ {\bf Case 2: $D = 2$.} We then have 
$$
d_2 = D = 2, \quad \nu \in \{2, 4\}, \quad t_2 \equiv t_4 \equiv 0 \bmod 2 \quad \text{ and } \quad t_1t_2+t_3t_4 \equiv 0 \bmod 4.
$$ 
Actually, the case $\nu = 4$ can never happen. Indeed, by the formula given in Lemma \ref{lemma4.1}, with the values $d_2 = D = 2$, $\nu =4$ and the constraints of the congruence modulo $4$ of the $t_i$, we see that $\Lambda(T_2^{-1} A_1 T_2) = \Z^2$. This is forbidden by Lemma \ref{nondegenerate}.
 
So we restrict to $\nu = 2$. The equations of the lattices $\Lambda(\sigma)$ ($\sigma \in \mathfrak S$) given in Lemma \ref{lemma4.1} are equivalent to a single equation 
\begin{equation*}
\begin{cases}
\Lambda({\rm id}) &: x_1\equiv 0 \bmod 2,\\
\Lambda(T_2^{-1} A_1 T_2) &: (t_1^2+t_3^2) x_1\equiv 0 \bmod 2, \\
\Lambda(T_2^{-1} A_2 T_2) &: (t_1^2 -t_3^2) x_1\equiv 0 \bmod 2,\\
\Lambda(T_2^{-1} A_3 T_2) &: x_1 \equiv 0 \bmod 2.
\end{cases}
\end{equation*}
Since $\Lambda(\sigma) \neq \Z^2$ by Lemma \ref{nondegenerate}, this implies that the coefficients of $x_1$ in the second and the third equations are odd. We deduce that these four equations show the equality 
$$
\Lambda(\sigma) = \{ (x_1, x_2) : x_1 \equiv 0 \bmod 2\} \text{ for all } \sigma \in \mathfrak S.
$$
This contradicts the covering \eqref{eD4cover}.

\subsubsubsection{ {If $t_2$ and $t_4$ are both odd} } 
The equality $t_1t_4-t_2t_3 = \pm 2$ implies that $t_1$ and $t_3$ have the same parity. We now split the argument according to this parity.

\paragraph{If $t_2$ and $t_4$ are both odd and if $t_1\equiv t_3 \equiv 0 \bmod 2$.} 
Since $d_2 = 2$, these conditions imply that $(t_1, t_3) \equiv (0,2)$ or $(2,0)$ modulo $4$. Therefore we have
$$
t_1t_2+t_3t_4 \equiv t_3t_4-t_1t_2 \equiv t_2t_3+t_1t_4 \equiv 2 \bmod 4.
$$
We implement these congruences into the first equations of \eqref{A1}, \eqref{A2} and \eqref{A3} to deduce the inclusions
$$
\Lambda( T_2^{-1} A_1 T_2), \, \Lambda(T_2^{-1} A_3 T_2) \subseteq \{(x_1, x_2) : x_1 +x_2 \equiv 0 \bmod 2\},
$$
and
$$
\Lambda(T_2^{-1} A_2 T_2) \subseteq \{ (x_1, x_2) : x_1 \equiv 0 \bmod 2\}.
$$
Recalling the inclusion
\begin{equation}
\label{easyinclusion}
\Lambda({\rm id}) \subseteq \{(x_1, x_2) : x_1 \equiv 0 \bmod 2\},
\end{equation}
and returning to the covering \eqref{eD4cover}, we obtain a covering of $\Z^2$ by two lattices with index $2$. This contradicts Theorem \ref{listofcoverings}.

\paragraph{If $t_2$ and $t_4$ are both odd and if $t_1\equiv t_3 \equiv 1 \bmod 2$.} 
Since $d_2 = 2$, the following congruences hold
$$
t_1t_2+t_3t_4 \equiv t_2t_3 + t_1 t_4 \equiv 0 \bmod 4 \quad \text{ and } \quad t_3t_4 -t_1t_2 \equiv 2 \bmod 4.
$$
We insert these congruences and the condition $4 \mid d_2D$ into the first equations of the systems \eqref{A1}, \eqref{A2} and \eqref{A3} to obtain the inclusions
$$
\Lambda(T_2^{-1} A_1 T_2), \Lambda(T_2^{-1} A_3 T_2) \subseteq \{(x_1, x_2) : x_2 \equiv 0 \bmod 2\},
$$
and
$$
\Lambda(T_2^{-1} A_2 T_2) \subseteq \{ (x_1, x_2) : x_1 \equiv 0 \bmod 2\}.
$$
These inclusions and the inclusion \eqref{easyinclusion} contradict \eqref{eD4cover}, since we would once more obtain a covering of $\Z^2$ by two lattices with index $2$.
 
Gathering all these cases, we proved that $d_2 \neq 2$. 

\subsubsection{ {The case $4 \mid d_2$ is impossible} }
\label{d2notdivisibleby4} 
We will suppose that $4 \mid d_2$ to arrive at a contradiction. This implies that $8 \mid d_2D$ thanks to \eqref{newsetforDnu}. We divide our proof according to the parity of $t_2$ and $t_4$.

\subsubsubsection{ {If $t_2$ and $t_4$ have different parities} } 
In that case, we have $\gcd(t_2^2 \pm t_4^2, 8) = 1$ and the first equations of \eqref{A1} and \eqref{A2} give the divisibility
$$
8 \mid [\Z^2 : \Lambda(T_2^{-1} A_i T_2)] \quad \text{ for } i \in \{1, 2\}.
$$
The covering \eqref{eD4cover} can not hold: we would obtain a covering with four proper lattices, with at least two with index divisible by $8$. This does not exist by Theorem \ref{listofcoverings}. 

\subsubsubsection{ {If $t_2$ and $t_4$ have the same parity} } 
Our first step is to show that we necessarily have
\begin{equation}
\label{t1 =t3bmod2}
t_1 \equiv t_3 \bmod 2.
\end{equation}
Suppose that this does not hold. We always have
$$
t_1^2 + t_3^2 \equiv t_1^2 - t_3^2 \bmod 2.
$$
We argue as follows:
 
\vskip .3cm
$\bullet$ If $t_1 \not\equiv t_3 \bmod 2$ and $t_2 \equiv t_4 \equiv 0 \bmod 2$. By \eqref{newsetforDnu}, we know that $2 \mid d_2D/\nu$. By the first equation of \eqref{IdD4} and the second equation of \eqref{A1} and \eqref{A2}, we deduce the inclusions
$$
\Lambda(\sigma) \subseteq \{ (x_1, x_2) : x_1 \equiv 0 \bmod 2\} \text{ for } \sigma \in \mathfrak S - \{T_2^{-1} A_3 T_2 \}.
$$
By \eqref{eD4cover}, we would obtain a covering by two proper lattices, and this contradicts Theorem~\ref{listofcoverings}. 
\vskip .3cm
$\bullet$ If $t_1 \not\equiv t_3 \bmod 2$ and $t_2 \equiv t_4 \equiv 1 \bmod 2$. We only study the first equations of \eqref{A1}, \eqref{A2} and \eqref{A3}. We use the congruences 
$$
t_1t_2+t_3t_4 \equiv t_3t_4 -t_1t_2 \equiv t_2t_3 +t_1t_4 \equiv 1 \bmod 2 \quad \text{ and } \quad 8 \mid d_2D
$$
to deduce that 
$$
8 \mid [\Z^2 : \Lambda(T_2^{-1} A_i T_2)] \quad \text { for } 1 \leq i \leq 3.
$$
The covering \eqref{eD4cover} then would be incompatible with Theorem \ref{listofcoverings}. So \eqref{t1 =t3bmod2} is proved. 

To summarize, we necessarily have the congruence conditions
$$
t_2 \equiv t_4 \bmod 2 \text{ and } t_1\equiv t_3 \bmod 2.
$$
Recall that the $t_i$ are coprime altogether, so modulo $2$, the quadruplet $(t_1, t_2, t_3, t_4)$ belongs to the following set $\Omega$ with three elements
$$
\Omega := \{(1,0,1, 0), (0,1,0,1), (1,1,1,1)\}.
$$
These possibilities will be the basis of our discussion below.

\paragraph{If $(t_1, t_2, t_3, t_4)\equiv (0,1,0,1) $ or $(1,1,1,1) \bmod 2$.} 
We then have the congruences $t_2^2+t_4^2 \equiv 2 \bmod 8$ and $2t_2t_4\equiv \pm 2 \bmod 8$. By considering the first equations in the systems \eqref{A1} and \eqref{A3}, we obtain the inequality 
$$
v_2([\Z^2 : \Lambda(T_2^{-1} A_i T_2)]) \geq 2 \text{ for } i \in \{1, 3\}.
$$
Furthermore, for this to be an equality, we must have $d_2D = 8$, which means $d_2 = 4$ and $D = 2$. 
 
When $(D, d_2) \neq (2, 4)$ (which is equivalent to $16 \mid d_2D$), the lattices $\Lambda(T_2^{-1}A_iT_2) $ for $i \in \{1, 3\}$ have index divisible by $8$. The covering \eqref{eD4cover} is therefore impossible thanks to Theorem \ref{listofcoverings}.
 
When $(D, d_2) = (2, 4)$, the congruences
$$
\pm d_2 = t_1t_4-t_2t_3 \equiv 4 \bmod 8,
$$
and $t_2 \equiv t_4 \equiv 1 \bmod 2$ imply 
$$
4 \equiv t_1t_4-t_2t_3 \equiv t_2t_4( t_1t_4-t_2t_3) \equiv t_1t_2-t_3t_4 \bmod 8.
$$
Since we also have $t_4^2-t_2^2\equiv 0 \bmod 8$, we conclude that
$$
\Lambda(T_2^{-1} A_2 T_2) \subseteq \{ (x_1, x_2): x_1 \equiv 0 \bmod 2\}
$$
by \eqref{A2} and by Lemma \ref{compare2adic}. Since $\Lambda( {\rm id}) $ satisfies the same inclusion (see \eqref{easyinclusion}) we obtain a contradiction with \eqref{eD4cover}, since we would obtain a covering of $\Z^2$ by three lattices with index $2$, $4$ and $4$.

\paragraph{If $(t_1, t_2, t_3, t_4)\equiv (1,0,1,0) \bmod 2$.} 
This case is more delicate. We handle this case by splitting the proof in three cases.

\vskip .3cm
$\diamond$ {\bf Case 1: $v_2(t_2) \neq v_2(t_4) \ (\geq 1)$.} We have the sequence of relations
$$
2\leq v_2(d_2) = \min (v_2(t_2), v_2(t_4) ) = v_2(t_1t_2 +t_3t_4) = v_2(t_3t_4-t_1t_2) = v_2(t_2t_3+t_1t_4),
$$
and the inequalities (recall that $D \in \{ 2, \, 4\}$)
$$
2 \min (v_2(t_2), v_2(t_4) ) = v_2(t_2^2 +t_4^2) = v_2(t_4^2-t_2^2) \geq v_2(d_2 D) \text{ and } v_2(2t_2t_4)\geq v_2(d_2D).
$$
With these remarks, we consider the first equations of the systems \eqref{IdD4}, \eqref{A1}, \eqref{A2} and \eqref{A3} to deduce the inclusions
$$
\Lambda(\sigma) \subseteq \{ (x_1, x_2) : x_1 \equiv 0 \bmod 2\} \text{ for all } \sigma \in \mathfrak S.
$$
These relations obviously contradict the covering \eqref{eD4cover}. 
 
Now we suppose that $v_2(t_2) = v_2(t_4) \ (\geq 1)$. Since, by hypothesis, the integers $t_1$ and $t_3$ are both odd, we have $v_2(d_2) > v_2(t_2) =v_2(t_4)$. So we split our forthcoming discussion according to the difference between $v_2(d_2)$ and $v_2(t_2) = v_2(t_4)$.

\vskip .3cm
$\diamond$ {\bf Case 2: $v_2(t_2) = v_2(t_4) \geq 1$ and $v_2(d_2) = v_2(t_2) +1$.} Under these hypotheses, we have
\begin{equation}
\label{2067}
v_2(t_1t_2 +t_3t_4) >v_2(d_2) \text{ and } v_2(t_2^2 +t_4^2) >v_2(d_2).
\end{equation}
For the first equality of \eqref{2067}, we use that
$$
t_1t_2 +t_3t_4 \equiv t_1t_4 +t_2 t_3 \equiv \pm d_2 +2 t_2 t_3 \equiv 0 \bmod 2d_2.
$$ 
For the second equation, we use the equality $v_2(t_2^2 +t_4^2) = 1 + 2v_2(t_2)$. By \eqref{newsetforDnu}, we have three possibilities
\begin{equation}
\label{2075}
(d_2, D, \nu) = (4, 2, 4), \quad v_2(d_2 D/\nu) > 1 \quad \text{ or } \quad D = 4.
\end{equation}
 
\vskip .3cm
\noindent $\bullet$ The first possibility is impossible because, with the above values of $d_2$, $D$ and $\nu$ and with the conditions on the $2$-adic valuations of the $t_i$, the two equations defining $\Lambda(T_2^{-1} A_1T_2)$ (see \eqref{A1}) are congruences modulo $8$, and all the coefficients of $x_1$ and $x_2$ are $\equiv 0 \bmod 8$. So we would have $\Lambda(T_2^{-1} A_1 T_2) = \Z^2$, which is contrary to Lemma \ref{nondegenerate}.

\vskip .3cm
\noindent $\bullet$ The second possibility of \eqref{2075} can not hold for the following reason: consider the second equations of the systems \eqref{A1} and \eqref{A3}. We benefit from the congruences
$$
t_1^2 +t_3 ^2 \equiv 2 t_1t_3 \equiv 2 \bmod 4
$$
to deduce the two inclusions
$$
\Lambda(T_2^{-1} A_i T_2) \subseteq \{ (x_1, x_2): x_1 \equiv 0 \bmod 2\} \quad \text{ for } i \in \{1, 3\}.
$$
Since the same inclusion holds for $\Lambda({\rm id})$ and since $\Lambda(T_2^{-1} A_2T_2) \neq \Z^2$, the covering \eqref{eD4cover} is impossible by Theorem \ref{listofcoverings}. 

\vskip .3cm
\noindent $\bullet$ We now prove that the third possibility of \eqref{2075} can not hold. Indeed, suppose that $D = 4$ and, as usual, let $g:= \gcd(t_2, t_4), \tilde t_2 := t_2 /g$ and $\tilde t_4 := t_4/g$. Since $g \mid d_2$, $g$ is necessarily a power of $2$ and we have $v_2(g) = v_2(t_2) = v_2(t_4) = v_2(d_2) - 1 \ (\geq 1)$ and the integers $\tilde t_2$ and $\tilde t_4$ are odd. We then have
$$
2 \equiv t_1 \tilde t_4 -\tilde t_2 t_3 \equiv (t_1t_3) (t_1\tilde t_4 - \tilde t_2 t_3) \equiv t_3 \tilde t_4 -t_1 \tilde t_2 \bmod 4,
$$
and therefore
\begin{equation}
\label{2095}
v_2(t_3t_4-t_1t_2) = v_2(d_2).
\end{equation}
Since $v_2(t_2) \geq 1$, it follows that
\begin{equation}
\label{2099}
t_4^2 -t_2^2 \equiv 0 \bmod d_2 D.
\end{equation}
To prove \eqref{2099}, write $t_2 = 2^t a_2$, $t_4 = 2^t a_4$ with odd $a_2$ and $a_4$ and $t \geq 1$. Then $t_4^2-t_2^2 =2^{2t} (a_4^2-a_2^2) \equiv 0 \bmod 2^{2t+2}$. Furthermore, $d_2 D = 2^{t+1} \cdot 4 = 2^{t+3}$. We obtain the desired congruence \eqref{2099}, since $2t+2 \geq t+3$ for $t \geq 1$. 
 
By the first equations of \eqref{IdD4} and \eqref{A2}, by \eqref{2095} and \eqref{2099} and by the hypothesis $D = 4$, we obtain the inclusions
$$
\Lambda({\rm id}),\ \Lambda(T_2^{-1} A_2 T_2) \subseteq \{ (x_1, x_2) : x_1 \equiv 0 \bmod 4\}.
$$
These inclusions are not compatible with the covering \eqref{eD4cover}, since we would obtain a covering of $\Z^2$ by three lattices containing one lattice with index $4$ (see Theorem \ref{listofcoverings}). Hence the case $D = 4$ does not happen. 

\vskip .3cm
 
$\diamond$ {\bf Case 3: $v_2(t_2) = v_2(t_4) \geq 1$ and $v_2(d_2) > v_2(t_2) +1$.} We will first prove the two equalities
\begin{equation}
\label{1972}
v_2(t_1t_2+t_3t_4) = v_2(t_2) + 1 \quad \text{ and } \quad v_2(t_2t_3 +t_1t_4) = v_2(t_2) + 1.
\end{equation}
\noindent $\bullet$ To prove the first equality, we write
$$
v_2(t_1t_2 + t_3t_4) = v_2(t_1t_2t_3 + t_3^2 t_4) = v_2(t_1 (t_1t_4 \pm d_2) + t_3^2t_4) = v_2((t_1^2 + t_3^2)t_4 \pm t_1d_2).
$$
We now observe that
$$
v_2(\pm t_1d_2) = v_2(d_2)
$$
and 
$$
v_2((t_1^2+t_3^2)t_4) = v_2(t_1^2 +t_3^2) +v_2(t_4) = v_2(t_2) +1,
$$ 
because $t_1$ and $t_3$ are both odd. Therefore we conclude that
$$
v_2(t_1t_2 + t_3t_4) = v_2((t_1^2 + t_3^2)t_4 \pm t_1d_2) = v_2(t_2) +1,
$$
since $v_2(t_2) +1 < v_2(d_2) $ by assumption. 
  
\noindent $\bullet$ To prove the second equality of \eqref{1972}, we write
$$
v_2(t_2t_3+t_1t_4) = v_2(2t_2t_3 +t_1t_4 -t_2t_3) = v_2(2t_2t_3 \pm d_2) = v_2(2t_2t_3) = v_2(2t_2) = v_2(t_2) +1.
$$
The proof of \eqref{1972} is complete. 

We now use \eqref{1972} to deduce from the first equations of \eqref{A1} and \eqref{A3} the two inclusions
\begin{equation}
\label{2140}
\Lambda(T_2^{-1} A_i T_2) \subseteq \{ (x_1, x_2) : x_1 \equiv 0 \bmod 2\} \quad \text{ for } i \in \{1, 3\}.
\end{equation}

\vskip .2cm
\noindent $\bullet$ Proof of \eqref{2140} for $i = 1$. The coefficient of $x_1$ in the first equation of \eqref{A1} satisfies $v_2(t_1t_2+t_3t_4 ) = v_2(t_2) +1 < v_2(d_2) < v_2(d_2 D)$. Furthermore, the coefficient of $x_2$ satisfies $v_2(t_2^2+t_4^2) = 1+2 v_2(t_2) > v_2(t_2) +1 = v_2(t_1t_2 + t_3t_4)$. Lemma \ref{compare2adic} gives the inclusion \eqref{2140} when $i = 1$. 

\vskip .2cm
\noindent $\bullet$ Proof of \eqref{2140} for $i = 3$. The coefficient of $x_1$ in the first equation of \eqref{A3} satisfies $v_2(t_2t_3+t_1t_4) = v_2(t_2) +1 < v_2(d_2) < v_2(d_2 D)$. Furthermore, the coefficient of $x_2$ satisfies $v_2(2t_2t_4) = 2 v_2(t_2) +1 > v_2(t_2) +1 = v_2(t_2t_3+t_1t_4)$. Lemma \ref{compare2adic} gives the inclusion \eqref{2140} when $i = 3$.

The proof of \eqref{2140} is complete. Combining the inclusions \eqref{2140} with the inclusion $\Lambda({\rm id}) \subseteq \{(x_1, x_2) : x_1 \equiv 0 \bmod 2\}$ and with the covering \eqref{eD4cover}, we arrive at a covering of $\Z^2$ by two proper lattices, which does not exist by Theorem \ref{listofcoverings}.

We investigated all the cases to assert that $4 \nmid d_2$. Gathering the properties of $d_2$ proved in \S \ref{nooddprime}, \S \ref{d2differentfrom1}, \S \ref{d2differentfrom2} and \S \ref{d2notdivisibleby4}, we see that $d_2$ does not exist. 

The proof of Theorem \ref{centralD4} is complete.


\begin{thebibliography}{19}








\bibitem{FKC3} 
\'E.~Fouvry and P.~Koymans,
\emph{Binary forms with the same value set I.}
Preprint.

\bibitem{FKD3D6} 
\'E.~Fouvry and P.~Koymans,
\emph{Binary forms with the same value set III. The case of ${\bf D}_3$ and ${\bf D}_6$.}
Preprint.



 

 




\bibitem{SX}
C. L. Stewart and S. Y. Xiao,
\emph{On the representation of integers by binary forms.} 
Math. Ann. {\bf 375} (2019), no. 1--2, 133--163.

 
\end{thebibliography}
\end{document}